\documentclass{amsart}
\usepackage{amsmath,amsthm,amssymb}
\usepackage[english]{babel}

\def\CC{\mathbb C} 
\def\RR{\mathbb R}
\def\HH{\mathbb H}
\def\Sp{\mathbb S}
\def\db{\overline\partial} 
\def\OO{\Omega}
\def\bO{\partial\Omega}
\def\cO{\overline\Omega}

\def\dd#1#2{\dfrac{\partial#1}{\partial#2}} 
\def\pr{$\psi$-regular}
\theoremstyle{theorem}
\newtheorem{theorem}{Theorem}

\newtheorem{proposition}[theorem]{Proposition}

\theoremstyle{remark}
\newtheorem{remark}{Remark}

\theoremstyle{example}
\newtheorem{example}{Example}

\begin{document}

\title[Holomorphic functions and  regular quaternionic functions]{Holomorphic functions and  regular quaternionic functions on the hyperk\" ahler space 
$\HH$}

\author{A. PEROTTI}
\thanks{\uppercase{W}ork partially supported by \uppercase{MIUR (PRIN P}roject ``\uppercase{P}ropriet\`a geometriche delle variet\`a reali e com\-ples\-se") and \uppercase{GNSAGA} of \uppercase{INdAM}}

\address{Department of Mathematics\\
  University of Trento\\ Via Sommarive, 14\\ I-38050 Povo Trento ITALY\\
E-mail: perotti@science.unitn.it}

\maketitle

\begin{abstract}{Let $\HH$ be the space of quaternions, with its standard hypercomplex structure.
  Let $\mathcal R(\Omega)$ be the module of \emph{$\psi$-regular} functions on $\Omega$. 
For every unitary vector $p$ in $\Sp^2\subset\HH$, $\mathcal R(\Omega)$ contains the space  of holomorphic functions w.r.t.\ the complex structure $J_p$ induced by $p$.
We prove the existence, on any bounded domain $\Omega$, of $\psi$-regular functions that are not $J_p$-holomorphic for any $p$. Our starting point is a result of Chen and Li concerning maps between hyperk\"ahler manifolds, where a similar result is obtained for a less restricted class of quaternionic maps. We give a criterion, based on the energy-minimizing property of holomorphic maps, that distinguishes $J_p$-holomorphic functions among $\psi$-regular functions.}
\end{abstract}

\parindent=18pt
{\small{\bfseries Key words:} Quaternionic  regular functions, hypercomplex structure, hyperk\"ahler space}

\parindent=12pt

\section{Introduction}
\label{sec:Introduction}
Let $\HH$ be the space of quaternions, with its standard hypercomplex structure given by the complex structures $J_1,J_2$ on $T\HH\simeq\HH$ defined by left multiplication by $i$ and $j$. Let  
  $J_1^*,J_2^*$ be the dual structures on $T^*\HH$.  
  
  We consider  the module $\mathcal R(\Omega)=\{f=f_1+f_2j\ |\ \overline\partial f_1=J_2^*(\partial\overline{f_2})\textrm{ on }\Omega\}$ of left \emph{$\psi$-regular} functions on $\Omega$. 
  These functions are in a simple correpondence with Fueter left regular functions, since they can be obtained from them by means of a real coordinate reflection in $\HH$.
 They have been studied by many authors (see for instance Sudbery\cite{Su}, Shapiro and Vasilevski\cite{SV} and N\=ono\cite{No}). \
The space $\mathcal R(\OO)$ contains the identity mapping and any holomorphic mapping $(f_1,f_2)$ on $\Omega$ defines a \pr\ function $f=f_1+f_2j$. This is no more true if we replace the class of $\psi$-regular functions with that of regular functions. The definition of $\psi$-regularity is also equivalent to that of \emph{$q$-holomorphicity} given by Joyce\cite{J} in the setting of hypercomplex manifolds. 

For every unitary vector $p$ in $\Sp^2\subset\HH$, $\mathcal R(\Omega)$ contains the space  $Hol_p(\Omega,{\HH})=\{f:\Omega\rightarrow{\HH}\ |\ df+pJ_p(df)=0 \textrm{ on }\Omega\}$ of holomorphic functions w.r.t.\ the complex structure $J_p=p_1J_1+p_2J_2+p_3J_3$ on $\Omega$ and to the structure induced on $\HH$ by left-multiplication by $p$  (\emph{$J_p$-holomorphic} functions on $\Omega$). 

We show that on every domain $\OO$ there exist $\psi$-regular functions that are not $J_p$-holomorphic for any $p$. A similar result was obtained by Chen and Li\cite{CL} for the larger class of \emph{q-maps} between hyperk\"ahler manifolds. 

This result is a consequence of a criterion (cf.\ Theorem \ref{criterion}) of $J_p$-holomorphicity, which is obtained using the energy-minimizing property of \pr\ functions (cf.\ Proposition \ref{energy}) and ideas of Lichnerowicz\cite{L} and Chen and Li\cite{CL}. 

In Sec.\ \ref{sec:OtherApplications} we give some other applications of the criterion. In particular, we show that if $\OO$ is connected, then the intersection $Hol_p(\OO,\HH)\cap Hol_{p'}(\OO,\HH)$ ($p\ne\pm p'$) contains only affine maps. This result is in accord with what was proved by Sommese\cite{So} about \emph{quaternionic maps} (cf.\ Sec.\ \ref{sec:QuaternionicMaps} for definitions).
   
\section{Fueter-regular and $\psi$-regular functions}
\subsection{Notations and definitions}
\label{sec:NotationsAndDefinitions}
We identify   the space $\CC^2$ with the set $\HH$ of quaternions by means of the mapping that associates  the pair $(z_1,z_2)=(x_0+ix_1,x_2+ix_3)$ with the quaternion $q=z_1+z_2j=x_0+ix_1+jx_2+kx_3\in\HH$. Let   $\OO$  be a bounded domain in $\HH\simeq \CC^2$.
    A quaternionic function $f=f_1+f_2j\in C^1(\OO)$ is \emph{(left) {regular}} on $\OO$ (in the sense of Fueter) if \[{\mathcal D}f=\dd{f }{x_0}+i\dd{f}{x_1}+j\dd{f }{x_2}+k\dd{ f}{x_3}=0 \text{\quad on $\OO$.}\]
Given the   ``structural vector'' $\psi=(1,i,j,-k)$,  $f$ is called \emph{(left) {$\psi$-regular}}  on $\OO$ if \[{\mathcal D'}f=\dd{f }{x_0}+i\dd{f}{x_1}+j\dd{f }{x_2}-k\dd{ f}{x_3}=0 \text{\quad on $\OO$}.\]
   
We recall some properties of regular functions, for which we refer to the papers of Sudbery\cite{Su}, Shapiro and Vasilevski\cite{SV} and N\=ono\cite{No}: 
  \begin{enumerate}
  \item  
  $f$ is $\psi$-regular $\Leftrightarrow\quad \dd{f_1}{\bar z_1}=\dd{\overline {f_2}}{z_2},\quad \dd {f_1}{\bar z_2}=-\dd{\overline {f_2}}{z_1}$.	     		\item
  Every holomorphic map $(f_1,f_2)$ on $\OO$ defines a $\psi$-regular function $f=f_1+f_2j$.
  \item
  The complex components are both holomorphic or both non-holomorphic.
  \item
    Every regular or $\psi$-regular function is harmonic.
  \item
  If $\OO$ is pseudoconvex, every  complex harmonic function  is the complex component of a $\psi$-re\-gular function   on $\OO$.
  \item
  The space $\mathcal R(\OO)$ of \pr\ functions on $\OO$ is a \emph{right} $\HH$-module
  with integral representation formulas.
  \end{enumerate}

\subsection{q-holomorphic functions}
\label{sec:qHolomorphicFunctionsJoyce}
A definition equivalent to $\psi$-regularity has been given by Joyce\cite{J} in the setting of hypercomplex manifolds. Joyce introduced the module of \emph{q-holomorphic} functions  on a hypercomplex manifold.  On this module he defined a (commutative) product. 
A hypercomplex structure on the manifold  $\HH$ is given by the complex structures   
  $J_1,J_2$ on $T\HH\simeq\HH$ defined by left multiplication by $i$ and $j$. Let  
  $J_1^*,J_2^*$ be the dual structures on $T^*\HH$.   In complex coordinates 
  \[
  \begin{cases}
  J_1^*dz_1=i\, dz_1,\ &J_1^*dz_2=i\, dz_2\\ J_2^*dz_1=-d\bar z_2,\ &J_2^*dz_2={d\bar z_1}\\ 
  J_3^*dz_1=i\,d\bar z_2,\ &J_3^*dz_2=-i\,{d\bar z_1}
  \end{cases}\]
  where we make the choice $J_3^*=J_1^*J_2^*\Rightarrow J_3=-J_1J_2$.
   
A function $f$ is $\psi$-regular if and only if $f$ is \emph{q-holomorphic}, i.e. 
\[df+iJ_1^*(df)+jJ_2^*(df)+kJ_3^*(df)=0.\]

In complex components $f=f_1+f_2j$, we can rewrite the  equations of $\psi$-regularity as
\[\db f_1=J_2^*(\partial\overline f_2).\]

\section{Holomorphic maps}
\label{sec:HolomorphicMaps}
\subsection{Holomorphic functions w.r.t.\  a complex structure $J_p$}
\label{sec:HolomorphicFunctionsWRTAComplexStructureJP}
Let $J_p=p_1J_1+p_2J_2+p_3J_3$ be the complex structure  on $\HH$ defined by a unit imaginary quaternion $p=p_1i+p_2j+p_3k$ in the sphere $\Sp^2=\{p\in\HH\ |\ p^2=-1\}$. 
It is well-known that every complex structure   compatible with the standard hyperk\"ahler structure of $\HH$ is of this form.
    If $f=f^0+if^1:\OO\rightarrow\CC$ is a \emph{$J_p$-holomorphic} function,  i.e. $df^0=J_p^*(df^1)$ or, equivalently, $df+iJ_p^*(df)=0$,
  then $f$ defines a $\psi$-regular function $\tilde f=f^0+pf^1$ on $\OO$. 
 We can identify  $\tilde f$ with a holomorphic function
 \[ \tilde f:(\OO,J_p)\rightarrow(\CC_p,L_p)\]
  where $\CC_p=\langle 1,p\rangle$ is a copy of $\CC$ in $\HH$ and $L_p$ is the complex structure defined on $T^*\CC_p\simeq\CC_p$  by left multiplication by $p$. 

More generally, we can consider the space of holomorphic maps from $(\OO,J_p)$ to $(\HH,L_p)$
\[Hol_p(\OO,\HH)=\{f:\OO\rightarrow\HH\ |\ \db_pf=0 \text{ on }\OO\}=Ker\db_p\]
 (the \emph{$J_p$-holomorphic maps} on $\OO$)
where $\db_p$ is the Cauchy-Riemann operator w.r.t.\  the structure $J_p$
\[\db_p=\frac12\left(d+pJ_p^*\circ d\right).\]

For any positive orthonormal basis $\{1,p,q,pq\}$ of $\HH$ ($p,q\in \Sp^2$), the equations of $\psi$-regularity can be rewritten in complex form as 
\[\db_p f_1=J_q^*(\partial_p\overline f_2),\]
 where $f=(f^0+pf^1)+(f^2+pf^3)q=f_1+f_2q$.
Then every $f\in Hol_p(\OO,\HH)$ is a \pr\ function on $\OO$.

\begin{remark}
\label{sec:properties}
1)	The \emph{identity} map is in $Hol_i(\OO,\HH)\cap Hol_j(\OO,\HH)$, but not in $Hol_k(\OO,\HH)$.
	
2) 	$Hol_{-p}(\OO,\HH)=Hol_p(\OO,\HH)$ 
	
3) If $f\in Hol_p(\OO,\HH)\cap Hol_{p'}(\OO,\HH)$, with $p\ne\pm p'$, then $f\in Hol_{p''}(\OO,\HH)$ for every $p''=\frac{\alpha p+\beta p'}{\|\alpha p+\beta p'\|}$.
	
4) $\psi$-regularity distinguishes between holomorphic and anti-holomorphic maps: if $f$ is an \emph{anti-holomorphic} map from $(\OO,J_p)$ to $(\HH,L_p)$, then $f$ can be \pr\ or not.
	For example, $f=\bar z_1+\bar z_2j\in Hol_j(\OO,\HH)\cap Hol_k(\OO,\HH)$ is a \pr\ function induced by the anti-holomorphic map \[(\bar z_1,\bar z_2):(\OO,J_1)\rightarrow (\HH,L_i),\] while	 
	 $(\bar z_1,0):(\OO,J_1)\rightarrow (\HH,L_i)$ induces the function $g=\bar z_1\notin \mathcal R(\OO)$.
\end{remark}

\subsection{Quaternionic maps}
\label{sec:QuaternionicMaps}
A particular class of $J_p$-holomorphic maps is constituted by the \emph{quaternionic maps} on the quaternionic manifold $\OO$. Sommese\cite{So} defined quaternionic maps between hypercomplex manifolds:   a quaternionic map is a map \[f:(X,J_1,J_2)\rightarrow (Y,K_1,K_2)\] that is holomorphic from $(X,J_1)$ to $(Y,K_1)$ \emph{and} from $(X,J_2)$ to $(Y,K_2)$.
	
	In particular, a quaternionic map \[f:(\OO,J_1,J_2)\rightarrow (\HH,J_1,J_2)\] is an element of $Hol_i(\OO,\HH)\cap Hol_j(\OO,\HH)$ and then a \pr\ function on $\OO$.	
	Sommese showed that  quaternionic maps are {affine}. They appear for example as transition functions for 4-dimensional \emph{quaternionic manifolds}.

\section{Non-holomorphic \pr\ maps}
\label{sec:NonHolomorphicPrMaps}
A natural question can now be raised: can $\psi$-regular maps  always be  made holomorphic by rotating the complex structure or do they constitute a new class of harmonic maps? In other words, does the space $\mathcal R(\OO)$ contain the union \[\bigcup_{p\in \Sp^2} Hol_p(\OO,\HH)\] properly?

Chen and Li\cite{CL} posed and answered the analogous question for the larger class of \emph{q-maps} between hyperk\"ahler manifolds. 
	In their definition, the complex structures of the source and target manifold can rotate \emph{independently}. This implies that also anti-holomorphic maps are q-maps.

\subsection{Energy and regularity}
  The \emph{energy} (w.r.t.\  the euclidean metric $g$) of a  map $f:\OO\rightarrow\CC^2\simeq\HH$, of class $C^1(\cO)$,  is the integral
  \[\mathcal E(f)=\frac12\int_\OO\|df\|^2dV=\frac12\int_\OO \langle g,f^*g\rangle dV=\frac12\int_\OO tr(J_\CC(f)
  \overline{J_\CC(f)}^T)dV,\]
  where $J_\CC(f)$ is the Jacobian matrix of $f$ with respect to the coordinates $\bar z_1,z_1,\bar z_2,z_2$.

Lichnerowicz\cite{L} proved that holomorphic maps between K\"ahler manifolds minimize the energy functional in their homotopy classes.
Holomorphic maps $f$ smooth on $\cO$ minimize energy in the homotopy class constituted by maps $u$ with $u_{|\bO} = f_{|\bO}$ which are homotopic to $f$ relative to $\bO$.

  From the theorem, functions $f\in Hol_p(\OO,\HH)$ minimize the energy functional in their homotopy classes (relative to $\bO$).  More generally:
  \begin{proposition}\label{energy}
  If $f$ is \pr\ on $\OO$, then it minimizes energy in its homotopy class (relative to $\bO$).
  \end{proposition}
  \begin{proof}We repeat arguments of Lichnerowicz, Chen and Li.  
  Let $i_1=i, i_2=j, i_3=k$ and let
  \[\mathcal K(f)=\int_\OO \sum_{\alpha=1}^3\langle J_\alpha,f^*L_{i_\alpha}\rangle dV,\quad 
  \mathcal I(f)=\frac12\int_\OO\|df+\sum_{\alpha=1}^3 L_{i_\alpha}\circ df\circ J_\alpha\|^2 dV.\]
  Then $\mathcal K(f)$ is a homotopy invariant of $f$ and $\mathcal I(f)=0$ if and only if $f\in\mathcal R(\OO)$. A computation similar to that made by Chen and Li\cite{CL} gives
  \[\mathcal E(f)+\mathcal K(f)=\frac14\mathcal I(f)\ge0.\]
  From this the result follows immediately.
	\end{proof}

\subsection{A criterion for holomorphicity}
  We now come to our main result. Let $f:\OO\rightarrow\HH$ be a function of class $C^1(\cO)$.
  
  \begin{theorem}\label{criterion}
  Let $A=(a_{\alpha\beta})$ be the $3\times 3$ matrix with entries 
  \[a_{\alpha\beta}=-\int_\OO \langle J_\alpha,f^*L_{i_\beta}\rangle dV.\] Then
  \begin{enumerate}
  \item
  {$f$ is \pr\ if and only if $\mathcal E(f)=tr A$.}
  \item
  {If $f\in\mathcal R(\OO)$, then $A$ is real, symmetric and \[tr A\ge\lambda_1=\max\{\text{eigenvalues of } A\}.\] It follows that $\det(A-(tr A)I_3)\le0$.}
  \item
  {If $f\in\mathcal R(\OO)$, then $f$ belongs to some space $Hol_p(\OO,\HH)$ if and only if $\mathcal E(f)=tr A=\lambda_1$ or, equivalently, $\det(A-(tr A)I_3)=0$.}
  \item
  {If $\mathcal E(f)=tr A=\lambda_1$,  $X_p=(p_1,p_2,p_3)$ is a unit eigenvector of $A$ relative to the largest eigenvalue $\lambda_1$ if and only if $f\in Hol_p(\OO,\HH)$.}
  \end{enumerate}
  \end{theorem}

\subsection{The existence of non-holomorphic \pr\ maps}
\label{sec:TheExistenceOfNonHolomorphicPrMaps}
The criterion can be applied to show that on every domain $\OO$ in $\HH$, there exist \pr\ functions that are not holomorphic.

  \begin{example}
  Let $f=z_1+z_2+\bar z_1+(z_1+z_2+\bar z_2)j$. Then $f$ is \pr, but not holomorphic, since  
  on the unit ball $B$ in $\CC^2$, $f$ has energy $\mathcal E(f)=6$ and the matrix $A$ of the theorem is
  	\[A=
\begin{bmatrix}
	2&0&0\\
	0&2&0\\
	0&0&2
\end{bmatrix}.\]Therefore $\mathcal E(f)=tr A>2=\lambda_1$.
    \end{example}
    
In the preceding example, the Jacobian matrix of the function has even rank, a necessary condition for a holomorphic map. In the case when the rank is odd, the non-holomorphicity follows immediately. For example,     
  $g=z_1+\bar z_1+\bar z_2j$  is \pr\ (on any $\OO$) but not $J_p$-holomorphic, for any $p$, since $rk J_\CC(f)$ is odd.
   
    \begin{example}
  The linear, \pr\ functions constitute a $\HH$-module of dimension 3 over $\HH$, generated e.g. by the set $\{z_1+z_2j,z_2+z_1j,\bar z_1+\bar z_2j\}$. An element \[f=(z_1+z_2j)q_1+(z_2+z_1j)q_2+(\bar z_1+\bar z_2j)q_3\] is holomorphic if and only if the  coefficients $q_1=a_1+a_2j$, $q_2=b_1+b_2j$, $q_3=c_1+c_2j$ satisfy the $6^{th}$-degree real homogeneous equation 
\[{\det(A-(tr A)I_3)=0}\]
  obtained after integration on $B$. The explicit expression of this equation is given in the Appendix. So   ``almost all'' (linear) \pr\ functions are non-holomorphic.
  \end{example}

\begin{example}
    A positive example (with $p\ne i,j,k$). Let $h=\bar z_1+(z_1+\bar z_2)j$. On the unit ball $h$ has energy $3$ and the matrix $A$ is 
    \[A=\left[\begin{array}{ccc}
	-1&0&2\\
	0&2&0\\
	2&0&2
	\end{array}\right]\]
		 then $\mathcal E(h)=tr A$ is equal to the (simple) largest eigenvalue, with unit eigenvector
	$X=\frac1{\sqrt5}(1,0,2)$.
		It follows that $h$ is $J_p$-holomorphic with $p=\frac1{\sqrt5}(i+2k)$, i.e. it satisfies the equation
	\[df+\tfrac15(i+2k)(J_1^*+2J_3^*)(df)=0.\]
	\
  \end{example}

  \begin{example}
  We give a quadratic example. Let $f=|z_1|^2-|z_2|^2+\bar z_1\bar z_2j$. $f$ has energy $2$ on $B$ and the matrix $A$ is 
    \[A=\begin{bmatrix}
	-2/3&0&0\\
	0&4/3&0\\
	0&0&4/3
	\end{bmatrix}\]
	Then $f$ is \pr\ but not holomorphic w.r.t.\  any complex structure $J_p$. 
  \end{example}

\subsection{Other applications of the criterion}
\label{sec:OtherApplications}

{ }1)
If $f\in Hol_p(\OO,\HH)\cap Hol_{p'}(\OO,\HH)$ for {two} $\RR$-independent $p,p'$, then $X_p,X_{p'}$ are independent eigenvectors relative to $\lambda_1$.
Therefore the eigenvalues of the matrix $A$ are  $\lambda_1=\lambda_2=-\lambda_3$.

If $f\in Hol_p(\OO,\HH)\cap Hol_{p'}(\OO,\HH)\cap Hol_{p''}(\OO,\HH)$ for {three} $\RR$-independent $p,p',p''$
then $\lambda_1=\lambda_2=\lambda_3=0\Rightarrow A=0$ and therefore $f$ has energy $0$ and $f$ is a (locally) \emph{constant} map.

2) If $\OO$ is connected, then 
   $Hol_p(\OO,\HH)\cap Hol_{p'}(\OO,\HH)$ ($p\ne\pm p'$) contains only \emph{affine} maps (cf. Sommese\cite{So}).
   
   We can assume  $p=i$, $p'=j$ since  in view of property 3) of Remark \ref{sec:properties} we can suppose $p$ and $p'$ orthogonal quaternions and then we can rotate the space of imaginary quaternions. Let $f\in Hol_i(\OO,\HH)\cap Hol_j(\OO,\HH)$ and $a=\left(\dd{f_1}{z_1},\dd{f_2}{z_1}\right)$, $b=\left(\overline{\dd{f_2}{z_2}},-\overline{\dd{f_1}{z_2}}\right)$. 
Since $f\in Hol_i(\OO,\HH)$, the matrix $A$ is obtained after integration on $\OO$ of the matrix
  \[\begin{bmatrix}
  |a|^2+|b|^2&0&0\\
  0&2Re\langle a,b\rangle&-2Im\langle a,b\rangle\\
  0&  -2Im\langle a,b\rangle&-2Re\langle a,b\rangle
  \end{bmatrix}\]
  where $\langle a,b\rangle$ denotes the standard hermitian product of $\CC^2$.
    
  Since $f\in Hol_j(\OO,\HH)$, we have $\int_\OO Im\langle a,b\rangle dV=0$ and $\int_\OO |a-b|^2 dV=0$. Therefore $a=b$ on $\OO$.
Then $a$ is holomorphic and anti-holomorphic w.r.t.\ the standard structure $J_1$. This means that $a$ is constant on $\OO$ and $f$ is an {affine} map with linear part of the form \[(a_1z_1-\bar a_2z_2)+(a_2z_1+\bar a_1z_2)j\] i.e. the right multiplication of $q=z_1+z_2j$ by the quaternion $a_1+a_2j$.

3)
We can give a classification of \pr\ functions based on the dimension of the set of complex structures w.r.t.\ which the function is holomorphic.
  Let $\OO$ be connected. Given a function $f\in\mathcal R(\OO)$, we set \[\mathcal J(f)=\{p\in \Sp^2\ |\ f\in Hol_p(\OO,\HH)\}.\]
    
  The space $\mathcal R(\OO)$ of \pr\ functions is the disjoint union of subsets of functions of the following four types:
  \renewcommand{\theenumi}{\roman{enumi}}
  \begin{enumerate}
  \item
    $f$ is $J_p$-holomorphic for three $\RR$-independent structures\\ $\Longrightarrow$ $f$ is a constant and $\mathcal J(f)=\Sp^2$.
  \item
    $f$ is $J_p$-holomorphic for exactly two $\RR$-independent structures\\ $\Longrightarrow$ $f$ is a \pr, invertible affine map and $\mathcal J(f)$ is an equator $S^1\subset \Sp^2$.
  \item
    $f$ is $J_p$-holomorphic for exactly one structure $J_p$ (up to sign of $p$) $\Longrightarrow \mathcal J(f)$ is a two-point set $S^0$.
   \item
    $f$ is \pr\ but not $J_p$-holomorphic w.r.t.\  any complex structure $\Longrightarrow \mathcal J(f)=\emptyset$.
  \end{enumerate}  


We will return in a subsequent paper to the application of the criterion  to the study of $\psi$-\emph{biregular} functions, which are invertible \pr\ functions with \pr\ inverse (see Kr\'olikowski and Porter\cite{KP} for the class of \emph{biregular} functions). This class contains as a proper subset the invertible holomorphic maps.

\section{Sketch of proof of Theorem \ref{criterion}}
\label{sec:SketchOfProofOfTheCriterion}
If $f\in\mathcal R(\OO)$, then $\mathcal E(f)=-\mathcal K(f)=tr A$.
Let \[\mathcal I_p(f)=\frac12\int_\OO\|df+ L_{p}\circ df\circ J_p\|^2 dV.\]
 Then we obtain, as in Chen and Li\cite{CL}
\[\mathcal E(f)+\int_\OO \langle J_p,f^*L_{p}\rangle dV=\frac14\mathcal I_p(f).\]
If $X_p=(p_1,p_2,p_3)$, then 
\[XAX^T=\sum_{\alpha,\beta}p_\alpha p_\beta a_{\alpha\beta}=-\int_\OO\langle \sum_\alpha p_\alpha J_\alpha,f^*\sum_\beta p_\beta L_{i_\beta}\rangle dV\]\[=-\int_\OO \langle J_p,f^*L_{p}\rangle dV=\mathcal E(f)-\frac14\mathcal I_p(f).\]

Then $tr A=\mathcal E(f)=XAX^T+\frac14\mathcal I_p(f)\ge XAX^T$, with equality if and only if $\mathcal I_p(f)=0$ i.e if and only if $f$ is a $J_p$-holomorphic map.


Let $M_\alpha$ ($\alpha=1,2,3$) be the matrix associated to $J^*_\alpha$ w.r.t.\  the basis $\{d\bar z_1,dz_1,d\bar z_2,dz_2\}$.
The entries of the matrix $A$ can be computed by the formula
\[a_{\alpha\beta}=-\int_\OO \langle J_\alpha,f^*L_{i_\beta}\rangle dV=
\frac12\int_\OO tr(\overline{B_\alpha}^T C_\beta)dV\]
where $B_\alpha=M_\alpha J_\CC(f)^T$ for $\alpha=1,2$, $B_\alpha=-M_\alpha J_\CC(f)^T$ for $\alpha=3$ and $C_\beta=J_\CC(f)^T M_\beta$ for $\beta=1,2,3$.

A direct computation shows how from the particular form of the Jacobian matrix of a \pr\ function it follows the symmetry property of $A$.

\section*{Appendix}
\label{sec:Appendix}
We give the explicit expression of the $6^{th}$-degree real homogeneous equation  satisfied by the complex coefficients of a linear $J_p$-holomorphic \pr\ function.

{$\frac1{16}det(A-(tr A)I_3)=a_1 a_2 b_2 c_1^2 \bar b_1 - a_1 a_2 b_1 c_1 c_2 \bar b_1 - 
    a_1^2 b_2 c_1 c_2 \bar b_1 + a_1^2 b_1 c_2^2 \bar b_1 - 
    a_1 c_1^2 \bar a_1 \bar b_1^2 - a_1 c_1 c_2 \bar a_2 \bar b_1^2 + 
    a_2^2 b_2 c_1^2 \bar b_2 - a_2^2 b_1 c_1 c_2 \bar b_2 - 
    a_1 a_2 b_2 c_1 c_2 \bar b_2 + a_1 a_2 b_1 c_2^2 \bar b_2 - 
    a_2 c_1^2 \bar a_1 \bar b_1 \bar b_2 - a_1 c_1 c_2 \bar a_1 \bar b_1 \bar b_2 - 
    a_2 c_1 c_2 \bar a_2 \bar b_1 \bar b_2 - a_1 c_2^2 \bar a_2 \bar b_1 \bar b_2 - 
    a_2 c_1 c_2 \bar a_1 \bar b_2^2 - a_2 c_2^2 \bar a_2 \bar b_2^2 + 
    a_1 a_2 b_1 b_2 c_1 \bar c_1 - a_1^2 b_2^2 c_1 \bar c_1 - 
    a_1 a_2 b_1^2 c_2 \bar c_1 + a_1^2 b_1 b_2 c_2 \bar c_1 - 
    2 a_1 b_1 c_1 \bar a_1 \bar b_1 \bar c_1 - a_1 b_2 c_1 \bar a_2 \bar b_1 \bar c_1 - 
    a_1 b_1 c_2 \bar a_2 \bar b_1 \bar c_1 - a_2 b_1 c_1 \bar a_1 \bar b_2 \bar c_1 - 
    2 a_1 b_2 c_1 \bar a_1 \bar b_2 \bar c_1 + a_1 b_1 c_2 \bar a_1 \bar b_2 \bar c_1 - 
    2 a_2 b_2 c_1 \bar a_2 \bar b_2 \bar c_1 + a_2 b_1 c_2 \bar a_2 \bar b_2 \bar c_1 - 
    a_1 b_2 c_2 \bar a_2 \bar b_2 \bar c_1 + c_1 \bar a_1 \bar a_2 \bar b_1 \bar b_2 \bar c_1 + 
    c_2 \bar a_2^2 \bar b_1 \bar b_2 \bar c_1 - c_1 \bar a_1^2 \bar b_2^2 \bar c_1 - 
    c_2 \bar a_1 \bar a_2 \bar b_2^2 \bar c_1 - a_1 b_1^2 \bar a_1 \bar c_1^2 - 
    a_1 b_1 b_2 \bar a_2 \bar c_1^2 + b_1 \bar a_1 \bar a_2 \bar b_2 \bar c_1^2 + 
    b_2 \bar a_2^2 \bar b_2 \bar c_1^2 + a_2^2 b_1 b_2 c_1 \bar c_2 - 
    a_1 a_2 b_2^2 c_1 \bar c_2 - a_2^2 b_1^2 c_2 \bar c_2 + 
    a_1 a_2 b_1 b_2 c_2 \bar c_2 - a_2 b_1 c_1 \bar a_1 \bar b_1 \bar c_2 + 
    a_1 b_2 c_1 \bar a_1 \bar b_1 \bar c_2 - 2 a_1 b_1 c_2 \bar a_1 \bar b_1 \bar c_2 + 
    a_2 b_2 c_1 \bar a_2 \bar b_1 \bar c_2 - 2 a_2 b_1 c_2 \bar a_2 \bar b_1 \bar c_2 - 
    a_1 b_2 c_2 \bar a_2 \bar b_1 \bar c_2 - c_1 \bar a_1 \bar a_2 \bar b_1^2 \bar c_2 - 
    c_2 \bar a_2^2 \bar b_1^2 \bar c_2 - a_2 b_2 c_1 \bar a_1 \bar b_2 \bar c_2 - 
    a_2 b_1 c_2 \bar a_1 \bar b_2 \bar c_2 - 2 a_2 b_2 c_2 \bar a_2 \bar b_2 \bar c_2 + 
    c_1 \bar a_1^2 \bar b_1 \bar b_2 \bar c_2 + 
    c_2 \bar a_1 \bar a_2 \bar b_1 \bar b_2 \bar c_2 - a_2 b_1^2 \bar a_1 \bar c_1 \bar c_2 - 
    a_1 b_1 b_2 \bar a_1 \bar c_1 \bar c_2 - a_2 b_1 b_2 \bar a_2 \bar c_1 \bar c_2 - 
    a_1 b_2^2 \bar a_2 \bar c_1 \bar c_2 - b_1 \bar a_1 \bar a_2 \bar b_1 \bar c_1 \bar c_2 - 
    b_2 \bar a_2^2 \bar b_1 \bar c_1 \bar c_2 - b_1 \bar a_1^2 \bar b_2 \bar c_1 \bar c_2 - 
    b_2 \bar a_1 \bar a_2 \bar b_2 \bar c_1 \bar c_2 - a_2 b_1 b_2 \bar a_1 \bar c_2^2 - 
    a_2 b_2^2 \bar a_2 \bar c_2^2 + b_1 \bar a_1^2 \bar b_1 \bar c_2^2 + 
    b_2 \bar a_1 \bar a_2 \bar b_1 \bar c_2^2=0$}

\end{document}